\documentclass{amsart}
\usepackage{mathrsfs}
\usepackage{mathrsfs}
\usepackage{amsfonts}
\usepackage{cases}
\usepackage{latexsym}
\usepackage{amsmath}
\usepackage[arrow,matrix]{xy}
\usepackage{stmaryrd}
\usepackage{amsfonts}
\usepackage{amsmath,amssymb,amscd,bbm,amsthm,mathrsfs,dsfont}

\usepackage{fancyhdr}
\usepackage{amsxtra,ifthen}
\usepackage{verbatim}

\numberwithin{equation}{section}

\theoremstyle{plain}
\newtheorem{theorem}{Theorem}[section]
\newtheorem{lemma}[theorem]{Lemma}

\theoremstyle{definition}

\begin{document}

\title[Nonlinear Lie-type Derivations of Von Neumann Algebras]
{Nonlinear Lie-type Derivations of Von Neumann Algebras}

\author{Zhankui Xiao, Zengqiang Lin and Feng Wei}

\address{Xiao: School of Mathematical Sciences, Huaqiao University,
Quanzhou, Fujian, 362021, P. R. China}

\email{zhkxiao@gmail.com}

\address{Lin: School of Mathematical Sciences, Huaqiao University,
Quanzhou, Fujian, 362021, P. R. China}

\email{lzq134@163.com}

\address{Wei: School of Mathematics, Beijing Institute of Technology,
Beijing, 100081, P. R. China}

\email{daoshuo@hotmail.com}

\begin{abstract}
Let $A$ be a von Neumann algebra with no central summands of type
$I_1$. We will show that every nonlinear Lie $n$-derivation on $A$
is of the standard form, i.e. it can be expressed as a sum of an
additive derivation and a central-valued mapping which annihilates
each $(n-1)$-th commutator of $A$.
\end{abstract}



%\thanks{The first author of this work is supported by a research
%foundation of Huaqiao University (Grant No. 10BS323). The second
%author of this work is partially supported by the Tianyuan Mathematics
%Foundation of China (Grant No. 11126331),
%and the Youth National Natural Science Foundation of China (Grant No.
%11101084).}

\maketitle
%\tableofcontents

\section{Introduction}\label{xxsec1}

Let $\mathcal{R}$ be a commutative ring with identity and
$\mathcal{A}$ be a unital associative algebra over $\mathcal{R}$. A
mapping (without the additivity or $\mathcal{R}$-linearity
assumption) $\varphi: \mathcal{A}\longrightarrow \mathcal{A}$ is
called a \textit{nonlinear Lie derivation} if
$$
\varphi([x, y])=[\varphi(x), y]+[x, \varphi(y)]
$$
for all $x,y\in \mathcal{A}$, a \textit{nonlinear Lie triple
derivation} if
$$
\varphi([[x, y], z])=[[\varphi(x), y], z]+[[x, \varphi(y)], z]+[[x,
y], \varphi(z)]
$$
for all $x,y,z\in \mathcal{A}$. Obviously, every nonlinear Lie
derivation is a nonlinear Lie triple derivation. But the converse
statement is not true in general. For instance, suppose that $d:
\mathcal{A}\longrightarrow \mathcal{A}$ is an additive derivation
and $f: \mathcal{A}\longrightarrow \mathcal{Z_A}$ is a mapping from
$\mathcal{A}$ into its center $\mathcal{Z_A}$ such that $f([[x, y],
z])=0$ for all $x, y, z\in \mathcal{A}$, then the mapping
$\varphi=d+f$ is a nonlinear Lie triple derivation of $\mathcal{A}$
and is not a nonlinear Lie derivation of $\mathcal{A}$. Taking into
account the previous two definitions of nonlinear Lie-type mappings,
we can extend them in one much more general way (see \cite{Abdullaev}). Suppose that $n\geq
2$ is a fixed positive integer. Let us see a sequence of polynomials
$$
\begin{aligned}
p_1(x_1)&=x_1\\
p_2(x_1,x_2)&=[p_1(x_1),x_2]=[x_1,x_2]\\
p_3(x_1,x_2,x_3)&=[p_2(x_1,x_2),x_3]=[[x_1,x_2],x_3]\\
p_4(x_1,x_2,x_3,x_4)&=[p_3(x_1,x_2,x_3),x_4]=[[[x_1,x_2],x_3],x_4]\\
\cdots\cdots\cdots &\cdots\cdots\cdots\cdots\\
p_n(x_1,x_2,\cdots,x_n)&=[p_{n-1}(x_1,x_2,\cdots,x_{n-1}),x_n].
\end{aligned}
$$
The polynomial $p_n(x_1,x_2,\cdots,x_n)$ is said to be an
$(n-1)$-\textit{th commutator} ($n\geq 2$). A mapping $\varphi:
A\longrightarrow A$ is called a \textit{nonlinear Lie
$n$-derivation} if
$$
\varphi(p_n(x_1,x_2,\cdots,x_n))=\sum_{i=1}^n
p_n(x_1,\cdots,x_{i-1}, \varphi(x_i),x_{i+1},\cdots,x_n)
$$
for all $x_1,x_2,\cdots,x_n\in \mathcal{A}$, an
\textit{$\mathcal{R}$-linear Lie $n$-derivation} if the mapping
$\varphi$ is $\mathcal{R}$-linear. Every (non-)linear Lie derivation
is a (non-)linear Lie $2$-derivation and every (non-)linear Lie
triple derivation is a (non-)linear Lie $3$-derivation. Furthermore,
if $d: \mathcal{A}\longrightarrow \mathcal{A}$ is an additive
derivation and that $f: \mathcal{A}\longrightarrow \mathcal{Z_A}$ is
a mapping from $\mathcal{A}$ into its center $\mathcal{Z_A}$ such
that $f(p_n(x_1,x_2,\cdots,x_n))=0$ for all $x_1, x_2, \cdots,
x_n\in \mathcal{A}$ ($n\geq 2$). Then the mapping
$$
\varphi=d+f \eqno(\clubsuit)
$$
is a nonlinear Lie $n$-derivation of $\mathcal{A}$ ($n\geq 2$). But
it is not a (nonlinear) derivation of $\mathcal{A}$ in the case where
$f$ can not annihilate $\mathcal{A}$
($n\geq 2$). We shall say that a nonlinear Lie $n$-derivation
$\varphi$ of $\mathcal{A}$ is \textit{standard} if it can be
expressed as the preceding form $(\clubsuit)$ ($n\geq 2$).

In recent years, there has been an increasing interest in
investigating whether (non-)linear Lie-type derivations on
$C^*$-algebras, and on more general operator algebras are of the
standard form $(\clubsuit)$. Many authors have made essential
contributions to related topics, see all literature references
[1-23]. Miers initiated the study of linear Lie-type derivations of
von Neumann algebras in \cite{Miers2, Miers3}. He \cite{Miers2}
proved that every linear Lie derivation on a von Neumann algebra
$\mathcal{A}$ is of the standard form $(\clubsuit)$. Furthermore, he
extended this result to the case of Lie triple derivations and
showed that if $\mathcal{A}$ is a von Neumann algebra with no
central summands of type $I_1$, then every Lie triple derivation has
the standard form $(\clubsuit)$ \cite{Miers3}. 
Using Johnson's original ideas dealing with continuous Lie derivations
from \cite{Johnson} and the theory of functional identities, Alaminos et al.
extended Miers' result \cite{Miers2} to Lie derivations from von Neumann
algebras into Banach bimodules \cite{AlaminosBresarVillena}.
Mathieu and Villena
\cite{MathieuVillena} proved that every linear Lie derivation on
$C^*$-algebras has the standard form $(\clubsuit)$. Lu and his
students systematically studied (non-)linear Lie-type derivations of
various operator algebras in their elegant works \cite{Lu1, Lu2,
Lu3, LuLiu1, LuLiu2}. The involved operator algebras include the
algebras of bounded linear operators, CSL algebras,
$\mathcal{J}$-subspace lattice algebras, nest algebras, reflexive
algebras. Roughly speaking, every (non-)linear Lie derivation or Lie
triple derivation on these operator algebras has the standard form
$(\clubsuit)$.

After Bre\v{s}ar's landmark paper \cite{Bresar}, there are more and more works
dealing with Lie-type mappings form the algebraical point of view (see \cite{BaiDu, BresarMiers, Lu2, LuLiu2}).
Cheung \cite{Cheung} gave sufficient conditions which enables every linear Lie
derivation on a triangular algebra to be standard. Yu and Zhang then extended Cheung's
result to nonlinear Lie derivations \cite{YuZhang}. In \cite{XiaoWei2} the first and third author
obtained that every nonlinear Lie triple derivation of full matrix algebra $M_n(\mathcal{R})$
is of the standard form provided that $\mathcal{R}$ is $2$-torsion free.
The present paper mostly motivated by Miers' work \cite{Miers3} and the work of
Bai and Du \cite{BaiDu}. It is of independent interesting to point out that
the methods of this paper can be similarly used to nonlinear Lie $n$-derivations
of the algebras $B(X)$ of bounded linear operators (see \cite{LuLiu2}).

\section{Preliminaries}\label{xxsec2}

From now on until the end of this paper, we always assume that
$\mathcal{A}$ is a von Neumann algebra with no central summands of
type $I_1$ (i.e. with no abelian summands). We denote
$\mathcal{Z_A}$ the center of $\mathcal{A}$. If
$A=A^*\in\mathcal{A}$, the central core of $A$, denoted by
$\underline{A}$, is defined to be ${\rm sup}\{S\in\mathcal{Z_A}|
S=S^*\leqslant A\}$. Clearly, the central core of a projection $P$
is the largest central projection contained in $P$. For $A\in
\mathcal{A}$, the central carrier of $A$, denoted by $\overline{A}$,
is the intersection of all central projections $P$ such that $PA=A$.

Let $P$ and $Q$ be nonzero orthogonal projection in $\mathcal{A}$
with $P+Q=I$, $\overline{P}=\overline{Q}=I$ and
$\underline{P}=\underline{Q}=0$, (see \cite{Miers1} and the last
paragraph of P.58 in \cite{Miers3}). Let
$\mathcal{A}_{11}=\{PXP|X\in\mathcal{A}\}$,
$\mathcal{A}_{12}=\{PXQ|X\in\mathcal{A}\}$,
$\mathcal{A}_{21}=\{QXP|X\in\mathcal{A}\}$,
$\mathcal{A}_{22}=\{QXQ|X\in\mathcal{A}\}$. Then we may write
$\mathcal{A}=\mathcal{A}_{11}+\mathcal{A}_{12}+\mathcal{A}_{21}+\mathcal{A}_{22}$.
We collect several fundamental properties of von Neumann algebras in
the following lemma for convenience.

\begin{lemma}\label{xxsec2.1}
Let $\mathcal{A}$ be a von Neumann algebra with no central summands of type $I_1$.
\begin{enumerate}
\item[(1)] \cite[Lemma 1]{Miers3} If $A\in \mathcal{A}_{ij}$ and $AX=0$ for all
$X\in \mathcal{A}_{jk}$ with $1\leq i,j,k\leq 2$, then $A=0$.

\item[(2)] \cite[Lemma 5]{Miers1} If $A\in \mathcal{A}$ commutes with all
$X_{12}\in \mathcal{A}_{12}$ and all $X_{21}\in \mathcal{A}_{21}$,
then $A$ commutes with all $X_{11}\in \mathcal{A}_{11}$ and all
$X_{22}\in \mathcal{A}_{22}$, and hence $A\in\mathcal{Z_A}$.

\item[(3)] \cite[Lemma 14]{Miers1} $\mathcal{A}_{ii}\cap
\mathcal{Z_A}=\{0\}$, where $i=1,2$.

\item[(4)] \cite[Lemma 5]{BresarMiers} If $C\in \mathcal{Z_A}$ such that
$C\mathcal{A}\subseteq \mathcal{Z_A}$, then $C=0$.
\end{enumerate}
\end{lemma}

We now recall a result of general Banach algebras called the
Kleinecke-Shirokov theorem \cite{Kleinecke, Shirokov}. If $d$ is a continuous derivation on a
Banach algebra $\mathcal{B}$ and $a\in\mathcal{B}$ is such that
$d^2(a)=0$, then $d(a)$ is quasi-nilpotent (see also
\cite{MathieuMurphy}, for example). The following version of the
Kleinecke-Shirokov theorem is more directly for our current use.

\begin{lemma}\label{xxsec2.2}
Let $\mathcal{B}$ be a Banach algebra. If $a,b\in\mathcal{B}$ such
that $[[a,b],b]=0$, then $[a,b]$ is quasi-nilpotent.
\end{lemma}

The main result of this paper is

\begin{theorem}\label{xxsec2.3}
Let $\mathcal{A}$ be a von Neumann algebra with no central summands
of type $I_1$. Let $\varphi: \mathcal{A}\longrightarrow\mathcal{A}$
be a nonlinear Lie $n$-derivation. Then $\varphi$ is of the form
$d+f$, where $d$ is an additive derivation of $\mathcal{A}$ and $f$
is a central-valued mapping which annihilates each $(n-1)$-th
commutator of $A$.
\end{theorem}

\section{Proof of the Main Result}\label{xxsec3}

In this section, we will prove the main result Theorem
\ref{xxsec2.3} by a series of lemmas. Let $\mathcal{A}$ be a von
Neumann algebra with no central summands of type $I_1$ and $\varphi$
be a nonlinear Lie $n$-derivation of $\mathcal{A}$. It is clear that
every Lie derivation is a Lie $n$-derivation for $n\geq 3$.
Therefore, without loss of generality we assume $n\geq 3$ for
convenience.

\begin{lemma}\label{xxsec3.1}
If $[X,Y]\in \mathcal{Z_A}$ for $X,Y\in \mathcal{A}$,
then $[\varphi(X),Y]+[X,\varphi(Y)]\in \mathcal{Z_A}$.
\end{lemma}

\begin{proof}
First, we note that
$$
\varphi(0)=\varphi(p_n(0,0,\cdots,0))=0.
$$
For any $X_{1}, \cdots, X_{n-2}\in\mathcal{A}$, we have
$p_n(X,Y,X_1,\cdots,X_{n-2})=0$ (note that we have assumed $n\geq 3$). Applying $\varphi$ to the identity, we get
$$
[\cdots[[\varphi(X),Y]+[X,\varphi(Y)],X_{1}],\cdots, X_{n-2}]=0.
$$
By Lemma \ref{xxsec2.2}, $[\cdots[[\varphi(X),Y]+[X,\varphi(Y)],X_{1}],\cdots, X_{n-3}]$
is quasi-nilpotent and therefore, being central, is zero. A direct recursive procedure shows
that $[[\varphi(X),Y]+[X,\varphi(Y)],X_{1}]=0$. That is
$[\varphi(X),Y]+[X,\varphi(Y)]\in \mathcal{Z_A}$.
\end{proof}

For later use we give out a equivalent definition of Lie $n$-derivation. Define a sequence of polynomials
recursively by letting
$$
\begin{aligned}
q_1(x_1)&=x_1\\
q_2(x_1,x_2)&=[x_2,q_1(x_1)]=[x_2,x_1]\\
\cdots &\cdots\cdots\\
q_n(x_1,x_2,\cdots,x_n)&=[x_n,q_{n-1}(x_1,x_2,\cdots,x_{n-1})].
\end{aligned}
$$
Then the definition of Lie $n$-derivation deduces that
$$
\varphi(q_n(x_1,x_2,\cdots,x_n))=\sum_{i=1}^n q_n(x_1,\cdots,x_{i-1},\varphi(x_i),x_{i+1},\cdots,x_n) \eqno(1)
$$
for all $x_1,x_2,\cdots,x_n\in \mathcal{A}$. On the other hand, a mapping of $\mathcal{A}$ satisfying $(1)$
is also a (nonlinear) Lie $n$-derivation.

\begin{lemma}\label{xxsec3.2}
There is $T_{0}\in \mathcal{A}$ such that
$\varphi(P)-[P,T_{0}]\in \mathcal{Z_A}$.
\end{lemma}

\begin{proof}
It is clearly that $M_{12}=q_n(M_{12},P,\cdots,P)=[P,[P,\cdots,[P,[P, M_{12}]]\cdots]]$
for any $M_{12}\in\mathcal{A}_{12}$. Apply $\varphi$ to the identity, we get
$$\begin{aligned}
\varphi(M_{12})&=[\varphi(P),M_{12}]+[P,[\varphi(P),M_{12}]]+\cdots\\
&+[P,[P,\cdots,[P,[\varphi(P),M_{12}]]\cdots]]+[P,[P,\cdots,[P,[P,\varphi(M_{12})]]\cdots]].
\end{aligned} \eqno(2)
$$
Note that
$$P([P,[\varphi(P),M_{12}]])Q=P(P[\varphi(P),M_{12}]-[\varphi(P),M_{12}]P)Q=P[\varphi(P),M_{12}]Q,$$
$$P[P,\varphi(M_{12})]Q=P(P\varphi(M_{12})-\varphi(M_{12})P)Q=P\varphi(M_{12})Q.$$
Multiplying $P$ and $Q$ from the left and the right in the above
Eq.(2) respectively, we have
$$P\varphi(M_{12})Q=(n-1)P[\varphi(P),M_{12}]Q+P\varphi(M_{12})Q.$$
Therefore $P[\varphi(P),M_{12}]Q=0$. That is
$$M_{12}\varphi(P)Q=P\varphi(P)M_{12}. \eqno(3)$$
Similarly, we get
$$M_{21}\varphi(P)P=Q\varphi(P)P.\eqno(4)$$
Eq.(3) and Eq.(4) shows that
$$[P\varphi(P)P+Q\varphi(P)Q,M_{12}]=[P\varphi(P)P+Q\varphi(P)Q,M_{21}]=0.$$
It follows from Lemma \ref{xxsec2.1} (2) that $P\varphi(P)P+Q\varphi(P)Q\in
\mathcal{Z_A}$. Denote $T_{0}=P\varphi(P)Q-Q\varphi(P)P$, then
$\varphi(P)-[P, T_{0}]=P\varphi(P)P+Q\varphi(P)Q\in\mathcal{Z_A}$.
\end{proof}

Let $T_0$ be as in Lemma \ref{xxsec3.2}. The mapping defined by
$T\mapsto [T,T_0]$ for any $T\in\mathcal{A}$ is an inner derivation.
Therefore, from now on, we can assume without loss of generality
$\varphi(P)\in \mathcal{Z_A}$.

\begin{lemma}\label{xxsec3.3}
If $X\in \mathcal{A}_{ij}, i\neq j$, then $\varphi(X)\in
\mathcal{A}_{ij}$.
\end{lemma}

\begin{proof}
We only treat the case $i=1,j=2$, the other case can be treated
similarly. If $X\in \mathcal{M}_{12}$, then
$X=[P,[P,\cdots,[P,X]\cdots]]$. Let $\varphi(X)=\Sigma_{1\leq
i,j\leq 2}X_{ij}$, where $X_{ij}\in \mathcal{A}_{ij}$. Since $\varphi(P)\in\mathcal{Z_A}$, we have
$$\begin{aligned}
\varphi(X)=[P,[P,\cdots,[P,\varphi(X)]\cdots]]=X_{12}+(-1)^{n-1}X_{21}.
\end{aligned}$$
Hence it is enough to show $X_{21}=0$.

For any $Y_{12}\in \mathcal{A}_{12}$, we have $[Y_{12},X]=0$. Lemma \ref{xxsec3.1} implies
$$
Z=[\varphi(Y_{12}),X]+[Y_{12},\varphi(X)]=[\varphi(Y_{12}),X]+(-1)^{n-1}[Y_{12},X_{21}]\in \mathcal{Z_A}.
$$
Apply $\varphi$ to the identity
$q_n(X,P,\cdots,P,Y_{12})=[Y_{12},[P,[P,\cdots,[P,X]\cdots]]]=0$, we have
$$\begin{aligned}
&[\varphi(Y_{12}),X]+[Y_{12},X_{12}+(-1)^{n-2}X_{21}]=[\varphi(Y_{12}),X]+(-1)^{n-2}[Y_{12},X_{21}]\\
=&Z-(-1)^{n-1}[Y_{12},X_{21}]+(-1)^{n-2}[Y_{12},X_{21}]=Z+(-1)^{n-2}2[Y_{12},X_{21}]=0.
\end{aligned}$$ Hence
$[X_{21},Y_{12}]=(-1)^{n-2}\frac{1}{2}Z\in\mathcal{Z_A}$.
By Lemma \ref{xxsec2.2}, $[X_{21},Y_{12}]$ is central quasi-nilpotent,
and hence $[X_{21},Y_{12}]=X_{21}Y_{12}-Y_{12}X_{21}=0$. This implies
$X_{21}Y_{12}=0$ for all $Y_{12}\in \mathcal{A}_{12}$, so $X_{21}=0$
by Lemma \ref{xxsec2.1} (1).
\end{proof}

\begin{lemma}\label{xxsec3.4}
$\varphi(Q)\in\mathcal{Z_A}$.
\end{lemma}

\begin{proof}
For any $M_{12}\in \mathcal{A}_{12}$, we note that

if $n$ is even, then
$M_{12}=p_n(M_{12},P,\cdots,P,Q)=[[\cdots[[M_{12},P],P],\cdots,P],Q]$,

if $n$ is odd, then $M_{12}=p_n(P,M_{12},P,\cdots,P,Q)=[[\cdots [[P, M_{12}],P],\cdots,P],Q]$.\\
For both cases, since $\varphi(M_{12})\in\mathcal{A}_{12}$, we have
$$\begin{aligned}
\varphi(M_{12})&=[[\cdots[\varphi(M_{12}),P],\cdots,P],Q]+[M_{12},\varphi(Q)]\\
&=\varphi(M_{12})+[M_{12},\varphi(Q)].
\end{aligned}$$
So $[M_{12},\varphi(Q)]=0$. We can similarly prove that $[M_{21},\varphi(Q)]=0$ for any $M_{21}\in \mathcal{A}_{21}$.
Therefore $\varphi(Q)\in\mathcal{Z_A}$ follows from Lemma \ref{xxsec2.1} (2).
\end{proof}

\begin{lemma}\label{xxsec3.5}
If $X\in \mathcal{A}_{ii}$, then
$\varphi(X)\in\mathcal{A}_{ii}+\mathcal{Z_A}\ (i=1,2)$.
\end{lemma}

\begin{proof}
Assume that $X\in \mathcal{A}_{11}$. Write $L(X)=\Sigma_{1\leq
i,j\leq2}X_{ij}$, where $X_{ij}\in \mathcal{A}_{ij}$. Since
$p_n(X,P,\cdots,P)=[\cdots[[X,P],P],\cdots,P]=0,$ we obtain
$$[\cdots[[\varphi(X),P],P],\cdots,P]=X_{21}+(-1)^{n-1}X_{12}=0.$$
Hence $X_{12}=X_{21}=0$. Thus $\varphi(X)=X_{11}+X_{22}$.

For any $Y\in\mathcal{A}_{22},$We can similarly prove that
$\varphi(Y)=Y_{11}+Y_{22}$, where $Y_{ii}\in \mathcal{A}_{ii}$. Since
$[X,Y]=0$, we have from Lemma \ref{xxsec3.1}
$$Z=[\varphi(X),Y]+[X,\varphi(Y)]=[X_{22},Y]+[X,Y_{11}]\in\mathcal{Z_A}.$$
This implies that $[X_{22},Y]=QZ\in
Q\mathcal{Z}_{\mathcal{A}}=\mathcal{Z}_{\mathcal{A}_{22}}$. Thus
$[X_{22},Y]$ is central quasi-nilpotent in $\mathcal{A}_{22}$ and
hence is zero. So that $X_{22}\in \mathcal{Z}_{\mathcal{A}_{22}}$.
There exists $C\in \mathcal{Z}_{\mathcal{A}}$ such that
$X_{22}=QC=(I-P)C=-PC+C\in\mathcal{A}_{11}+\mathcal{Z}_{\mathcal{A}}$.
Hence
$\varphi(X)=X_{11}+X_{22}=X_{11}-PC+C\in\mathcal{A}_{11}+\mathcal{Z}_{\mathcal{A}}$.
Similarly $\varphi(Y)\in\mathcal{A}_{22}+\mathcal{Z_A}$ and this completes the proof of the lemma.
\end{proof}

\begin{lemma}\label{xxsec3.6}
Let $X\in \mathcal{A}$.
\begin{enumerate}
\item[(1)] If $n$ is even, then $\varphi(PXQ-QXP)=P\varphi(X)Q-Q\varphi(X)P$. If $n$ is odd, then
$\varphi(PXQ+QXP)=P\varphi(X)Q+Q\varphi(X)P$.

\item[(2)] If $PXQ=0$, then $P\varphi(X)Q=0$. If $QXP=0$, then $Q\varphi(X)P=0$.
\end{enumerate}
\end{lemma}

\begin{proof}
(1). It is easy to see that
$$
q_n(X,P,\cdots,P)=[P,[P,\cdots,[P,X]\cdots]]=PXQ+(-1)^{n-1}QXP.
$$
Applying $\varphi$ to the last equation we get the desired results.

(2). Assume $PXQ=0$. If $n$ is odd, there is
$\varphi(QXP)=\varphi(PXQ+QXP)=P\varphi(X)Q+Q\varphi(X)P\in \mathcal{A}_{21}$. Therefore
$P\varphi(X)Q=0$. The other cases can be proved similarly.
\end{proof}

\begin{lemma}\label{xxsec3.7}
For any $X\in\mathcal{A}$, $X_{12},Y_{12}\in\mathcal{A}_{12}$, we have
\begin{enumerate}
\item[(1)] $[\varphi(X+X_{12})-\varphi(X),Y_{12}]=0$,

\item[(2)] $\varphi(X+X_{12})-\varphi(X)=P(\varphi(X+X_{12})-\varphi(X))Q+Z$,
where $Z\in \mathcal{Z}_{\mathcal{A}}$.
\end{enumerate}
\end{lemma}

\begin{proof}
(1). Since $[X+X_{12},Y_{12}]=[X,Y_{12}]$, we have
$$[Y,[X+X_{12},[P,\cdots,[P,Y_{12}]\cdots]]]=[Y,[X,[P,\cdots,[P,Y_{12}]\cdots]]]$$
for all $Y\in\mathcal{A}$. Note that $\varphi(Y_{12})\in\mathcal{A}_{12}$.
Applying $\varphi$ to the equation, we have
$$[Y,[\varphi(X+X_{12}),Y_{12}]]+[Y,[X+X_{12},\varphi(Y_{12})]]=[Y,[\varphi(X),Y_{12}]]+[Y,[X,\varphi(Y_{12})]].$$
Then $[Y,[\varphi(X+X_{12})-\varphi(X),Y_{12}]]=0$. It in turn
implies that $[\varphi(X+X_{12})-\varphi(X),Y_{12}]$ is central quasi-nilpotent
in $\mathcal{A}$ and hence is zero.

(2). It follows from (1) and \cite[Lemma 2]{BaiDu}.
\end{proof}

\begin{lemma}\label{xxsec3.8}
For $1\leq i\neq j\leq 2$, we have
\begin{enumerate}
\item[(1)] $\varphi(X_{ii}+X_{ij})-\varphi(X_{ii})-\varphi(X_{ij})\in
\mathcal{Z}_{\mathcal{A}}$,

\item[(2)] $\varphi(X_{ii}+X_{ji})-\varphi(X_{ii})-\varphi(X_{ji})\in
\mathcal{Z}_{\mathcal{A}}$.
\end{enumerate}
\end{lemma}

\begin{proof} We only prove (1) for $i=1, j=2$.  The other cases can be proved similarly.
From Lemma \ref{xxsec3.7},
$$
\varphi(X_{11}+X_{12})-\varphi(X_{11})=P(\varphi(X_{11}+X_{12})-\varphi(X_{11}))Q+Z
$$
for some central element $Z\in\mathcal{Z}_{\mathcal{A}}$. It is
enough to show that $$\varphi(X_{12})=P(\varphi(X_{11}+X_{12})-\varphi(X_{11}))Q.$$
In fact, by Lemma \ref{xxsec3.6},
$$
\begin{aligned}
\varphi(X_{12})&=L(P(X_{11}+X_{12})Q\pm Q(X_{11}+X_{12})P)\\
&=P(\varphi(X_{11}+X_{12}))Q\pm Q(\varphi(X_{11}+X_{12}))P,
\end{aligned}
$$
where the sign depend on $n$ is even or odd. Then $\varphi(X_{12})\in \mathcal{A}_{12}$ and $\varphi(X_{11})\in
\mathcal{A}_{11}+\mathcal{Z_A}$ deduce that
$$
\varphi(X_{12})=P(\varphi(X_{11}+X_{12}))Q=P(\varphi(X_{11}+X_{12})-\varphi(X_{11}))Q.
$$
\end{proof}

\begin{lemma}\label{xxsec3.9}
$\varphi$ is additive on $\mathcal{A}_{12}$ and $\mathcal{A}_{21}$.
\end{lemma}

\begin{proof}
Let $X_{12}, Y_{12}\in\mathcal{M}_{12}$. Since
$$X_{12}+Y_{12}=[P+X_{12},Q+Y_{12}]=[\cdots[[[P+X_{12},Q+Y_{12}],Q],Q],\cdots,Q],$$
we have from Lemma \ref{xxsec3.8}
$$\begin{aligned}
&\varphi(X_{12}+Y_{12})\\
=&[\cdots[[[\varphi(P+X_{12}),Q+Y_{12}],Q],Q],\cdots,Q]\\
&+[\cdots[[[P+X_{12},\varphi(Q+Y_{12})],Q],Q],\cdots,Q]\\
=&[\cdots[[[\varphi(P)+\varphi(X_{12}),Q+Y_{12}],Q],Q],\cdots,Q]\\
&+[\cdots[[[P+X_{12},\varphi(Q)+\varphi(Y_{12})],Q],Q],\cdots,Q]\\
=&\varphi(X_{12})+\varphi(Y_{12}).
\end{aligned}$$
Similarly, $\varphi$ is additive on $\mathcal{A}_{21}$.
\end{proof}

\begin{lemma}\label{xxsec3.10}
$\varphi(X_{11}+X_{22})-\varphi(X_{11})-\varphi(X_{22})\in \mathcal{Z}_{\mathcal{A}}$.
\end{lemma}

\begin{proof}
Clearly $[X_{11}+X_{22},Y_{12}]=X_{11}Y_{12}-Y_{12}X_{22}$ for any
$Y_{12}\in \mathcal{A}_{12}$. By Lemma \ref{xxsec3.9},
$$\begin{aligned}
&\varphi(X_{11}Y_{12}-Y_{12}X_{22})=\varphi(X_{11}Y_{12}))+\varphi(-Y_{12}X_{22})\\
=& \varphi([\cdots[[[X_{11},Y_{12}],Q],Q],\cdots,Q])+\varphi([\cdots[[[X_{22},Y_{12}],Q],Q],\cdots,Q])\\
=&[\cdots[[[\varphi(X_{11}),Y_{12}],Q],Q],\cdots,Q]+[\cdots[[[X_{11},\varphi(Y_{12})],Q],Q],\cdots,Q]\\
&+[\cdots[[[\varphi(X_{22}),Y_{12}],Q],Q],\cdots,Q]+[\cdots[[[X_{22},\varphi(Y_{12})],Q],Q],\cdots,Q]\\
=&[\varphi(X_{11})+\varphi(X_{22}),Y_{12}]+[X_{11}+X_{22},\varphi(Y_{12})].
\end{aligned}\eqno(5)
$$
On the other hand, there is
$$[X,[X_{11}+X_{22},[P,\cdots,[P,Y_{12}]\cdots]]]=[X,[P,[P,\cdots,[P,X_{11}Y_{12}-Y_{12}X_{22}]\cdots]]]$$
for all $X\in\mathcal{A}$. Applying $\varphi$ to the above identity, we have from Eq. (5)
$$\begin{aligned}
&[X,[\varphi(X_{11}+X_{22}),Y_{12}]]+[X,[X_{11}+X_{22},\varphi(Y_{12})]]\\
=&[X,[P,[P,\cdots,[P,\varphi(X_{11}Y_{12}-Y_{12}X_{22})]\cdots]]]\\
=&[X,\varphi(X_{11}Y_{12}-Y_{12}X_{22})]\\
=&[X,[\varphi(X_{11})+\varphi(X_{22}),Y_{12}]]+[X,[X_{11}+X_{22},\varphi(Y_{12})]].
 \end{aligned}$$
Hence $[X,[\varphi(X_{11}+X_{22})-\varphi(X_{11})-\varphi(X_{22}), Y_{12}]]=0$.
It follows that $[\varphi(X_{11}+X_{22})-\varphi(X_{11})-\varphi(X_{22}), Y_{12}]$ is central
quasi-nilpotent and hence is zero. From \cite[Lemma 2]{BaiDu},
$\varphi(X_{11}+X_{22})-\varphi(X_{11})-\varphi(X_{22})\in
 \mathcal{A}_{12}+\mathcal{Z}_{\mathcal{A}}$.

However, since $P(X_{11}+X_{22})Q=0$, we have
$P(\varphi(X_{11}+X_{22}))Q=0$ by Lemma \ref{xxsec3.6} (2). So Lemma \ref{xxsec3.5} implies
$$L(X_{11}+X_{22})-L(X_{11})-L(X_{22})\in
\mathcal{Z}_{\mathcal{M}}.$$
\end{proof}

\begin{lemma}\label{xxsec3.11}
$\varphi(X_{ii}+Y_{ii})-\varphi(X_{ii})-\varphi(Y_{ii})\in \mathcal{Z_A}$ for $i=1,2$.
\end{lemma}

\begin{proof}
We only prove the case $i=1$ and the other case can be proved similarly. For $Y_{12}\in\mathcal{A}_{12}$,
$[X_{11}+Y_{11},Y_{12}]=X_{11}Y_{12}+Y_{11}Y_{12}$. By Lemma \ref{xxsec3.9},
$$\begin{aligned}
&\varphi(X_{11}Y_{12}+Y_{11}Y_{12})=\varphi(X_{11}Y_{12}))+\varphi(Y_{11}Y_{12})\\
=&\varphi([\cdots[[[X_{11},Y_{12}],Q],Q],\cdots,Q])+\varphi([\cdots[[[Y_{11},Y_{12}],Q],Q],\cdots,Q])\\
=&[\varphi(X_{11}),Y_{12}]+[X_{11},\varphi(Y_{12})]+[\varphi(Y_{11}),Y_{12}]+[Y_{11},\varphi(Y_{12})]\\
=&[\varphi(X_{11})+\varphi(Y_{11}),Y_{12}]+[X_{11}+Y_{11},\varphi(Y_{12})].
\end{aligned}\eqno(6)
$$
On the other hand, there is
$$[X,[X_{11}+Y_{11},[P,\cdots,[P,Y_{12}]\cdots]]]=[X,[P,[P,\cdots,[P,X_{11}Y_{12}+Y_{11}Y_{12}]\cdots]]]$$
for all $X\in\mathcal{A}$. Applying $\varphi$ to the above identity, we have from Eq. (6)
$$\begin{aligned}
&[X,[\varphi(X_{11}+Y_{11}),Y_{12}]]+[X,[X_{11}+Y_{11},\varphi(Y_{12})]]\\
=&[X,[P,[P,\cdots,[P,\varphi(X_{11}Y_{12}+Y_{11}Y_{12})]\cdots]]]\\
=&[X,\varphi(X_{11}Y_{12}+Y_{11}Y_{12})]\\
=&[X,[\varphi(X_{11})+\varphi(Y_{11}),Y_{12}]]+[X,[X_{11}+Y_{11},\varphi(Y_{12})]].
\end{aligned}
$$
Hence $[X,[\varphi(X_{11}+Y_{11})-\varphi(X_{11})-\varphi(Y_{11}), Y_{12}]]=0$.
It follows that $[\varphi(X_{11}+Y_{11})-\varphi(X_{11})-\varphi(Y_{11}), Y_{12}]$ is central
quasi-nilpotent and hence is zero. From Lemma \cite[Lemma 2]{BaiDu},
$\varphi(X_{11}+Y_{11})-\varphi(X_{11})-\varphi(Y_{11})\in
 \mathcal{A}_{12}+\mathcal{Z}_{\mathcal{A}}$.

However, since $P(X_{11}+Y_{11})Q=0$, we have
$P(\varphi(X_{11}+Y_{11}))Q=0$ by Lemma \ref{xxsec3.6} (2). So Lemma \ref{xxsec3.5} implies
$$\varphi(X_{11}+Y_{11})-\varphi(X_{11})-\varphi(Y_{11})\in\mathcal{Z}_{\mathcal{A}}.$$
\end{proof}

\begin{lemma}\label{xxsec3.12}
$\varphi(X_{ii}+X_{jj}+X_{ij})-\varphi(X_{ii})-\varphi(X_{jj})-\varphi(X_{ij})\in
\mathcal{Z}_{\mathcal{A}}$.
\end{lemma}

\begin{proof}
We only prove the case for $i=1,j=2$. From Lemmas \ref{xxsec3.10} and \ref{xxsec3.7}, we have
$$\begin{aligned}
& \varphi(X_{11}+X_{22}+X_{12})-\varphi(X_{11})-\varphi(X_{22})\\
=& \varphi(X_{11}+X_{22}+X_{12})-\varphi(X_{11}+X_{22})+Z_0\\
=& P(\varphi(X_{11}+X_{22}+X_{12})-\varphi(X_{11}+X_{22}))Q+Z\\
=& P\varphi(X_{11}+X_{22}+X_{12})Q+Z\in\mathcal{A}_{12}+\mathcal{Z}_{\mathcal{A}}.
\end{aligned}$$
for some central element $Z_0,Z\in\mathcal{Z}_{\mathcal{A}}$. Lemma \ref{xxsec3.6} shows that
$$\begin{aligned}
\varphi(X_{12})=&\varphi(P(X_{11}+X_{22}+X_{12})Q\pm Q(X_{11}+X_{22}+X_{12})P)\\
=&P\varphi(X_{11}+X_{22}+X_{12})Q\pm Q\varphi(X_{11}+X_{22}+X_{12})P\\
=&P\varphi(X_{11}+X_{22}+X_{12})Q.
\end{aligned}$$
Therefore
$$\varphi(X_{11}+X_{22}+X_{12})-\varphi(X_{11})-\varphi(X_{22})-\varphi(X_{12})=Z\in\mathcal{Z}_{\mathcal{A}}.$$
\end{proof}

\begin{lemma}\label{xxsec3.13}
$\varphi$ is almost additive on $\mathcal{A}$, i.e., for all
$X,Y\in\mathcal{A}$, $\varphi(X+Y)-\varphi(X)-\varphi(Y)\in
\mathcal{Z}_{\mathcal{A}}$.
\end{lemma}

\begin{proof}
By Lemma \ref{xxsec3.9} and Lemma \ref{xxsec3.11}, we only need to prove
$$\varphi(X_{11}+X_{12}+X_{21}+X_{22})-\varphi(X_{11})-\varphi(X_{12})-\varphi(X_{21})-\varphi(X_{22})\in
\mathcal{Z}_{\mathcal{A}},$$
where $X_{ij}\in\mathcal{A}_{ij}$.

From Lemma \ref{xxsec3.12} and Lemma \ref{xxsec3.7}, we have
$$\begin{aligned}
&[\varphi(X_{11}+X_{12}+X_{21}+X_{22})-\varphi(X_{11})-\varphi(X_{12})-\varphi(X_{21})-\varphi(X_{22}),Y_{12}]\\
&=[\varphi(X_{11}+X_{12}+X_{21}+X_{22})-\varphi(X_{11})-\varphi(X_{21})-\varphi(X_{22}),Y_{12}]\\
&=[\varphi(X_{11}+X_{12}+X_{21}+X_{22})-\varphi(X_{11}+X_{21}+X_{22}),Y_{12}]\\
&=0
\end{aligned}$$
for all $Y_{12}\in\mathcal{A}_{12}$.
Similarly, we have
$$[\varphi(X_{11}+X_{12}+X_{21}+X_{22})-\varphi(X_{11})-\varphi(X_{12})-\varphi(X_{21})-\varphi(X_{22}),Y_{21}]=0$$
for all $Y_{21}\in\mathcal{A}_{21}$. Now Lemma \ref{xxsec2.1} (2)
asserts that
$$\varphi(X_{11}+X_{12}+X_{21}+X_{22})-\varphi(X_{11})-\varphi(X_{12})-\varphi(X_{21})-\varphi(X_{22})\in
\mathcal{Z}_{\mathcal{A}}.$$
\end{proof}

Now we are at the position to prove our main result.
\vspace{6pt}

{\noindent {\bf Proof of Theorem 2.3.}}
From Lemma \ref{xxsec3.3} and Lemma \ref{xxsec3.5}, we know that if
$X_{ij}\in \mathcal{A}_{ij}$ with $i\neq j$, then
$\varphi(X_{ij})=Y_{ij}\in\mathcal{A}_{ij}$; if $X_{ii}\in
\mathcal{A}_{ii}$, then $\varphi(X_{ii})=Y_{ii}+Z$, where
$Y_{ii}\in\mathcal{A}_{ii}$ and $Z\in \mathcal{Z}_{\mathcal{A}}$ are
unique determined by Lemma \ref{xxsec2.1} (3).
Therefore it is reasonable to
define a mapping $d:\mathcal{A}\rightarrow \mathcal{A}$ by
$d(X_{11}+X_{12}+X_{21}+X_{22})=Y_{11}+Y_{12}+Y_{21}+Y_{22}$. It is
clear that $\varphi(X)-d(X)\in\mathcal{Z}_{\mathcal{A}}$, so we can define
a mapping $f:\mathcal{A}\rightarrow\mathcal{Z}_{\mathcal{A}}$ by
$f(X)=\varphi(X)-d(X)$.

\textbf{Step 1.} We prove that $d$ is additive.

By Lemma \ref{xxsec3.10}, $d$ is additive on $\mathcal{A}_{12}$ and
$\mathcal{A}_{21}$. We claim that $d$ is also additive on
$\mathcal{A}_{ii},i=1,2$. In fact, for any
$X_{ii},Y_{ii}\in\mathcal{M}_{ii}$, we have
$$\begin{aligned}
&d(X_{ii}+Y_{ii})-d(X_{ii})-d(Y_{ii})\\
=&\varphi(X_{ii}+Y_{ii})-f(X_{ii}+Y_{ii})-\varphi(X_{ii})+f(X_{ii})-\varphi(Y_{ii})+f(Y_{ii})\in\mathcal{A}_{ii}\cap\mathcal{Z}_{\mathcal{A}}.
\end{aligned}$$
By Lemma \ref{xxsec2.1} (3), we get
$d(X_{ii}+Y_{ii})-d(X_{ii})-d(Y_{ii})=0.$

Assume that $X=\sum_{1\leq i,j\leq 2}X_{ij}$, $Y=\sum_{1\leq
i,j\leq2}Y_{ij}$, where $X_{ij},Y_{ij}\in\mathcal{A}_{ij}$. By
definition of $d$, we have
$$\begin{aligned}
d(X+Y)&=d(\sum_{1\leq i,j\leq2}(X_{ij}+Y_{ij}))=\sum_{1\leq i,j\leq2}d(X_{ij}+Y_{ij}) \\
&=\sum_{1\leq i,j\leq2}(d(X_{ij})+d(Y_{ij}))=d(X)+d(Y).
\end{aligned}$$

\textbf{Step 2.} We prove that $d$ is a derivation.

Assume that $X_{ij},Y_{ij},M_{ij}\in\mathcal{A}_{ij}$ with $1\leq
i,j\leq 2$. If $i\neq j$, we have
$$\begin{aligned}
d(X_{ii}Y_{ij})&=\varphi(X_{ii}Y_{ij})=\varphi([\cdots[[X_{ii},Y_{ij}],P_{j}],\cdots,P_{j}])\\
&=[\cdots[[\varphi(X_{ii}),Y_{ij}]+[X_{ii},\varphi(Y_{ij})],P_{j}],\cdots,P_{j}]\\
&=[\cdots[d(X_{ii})Y_{ij}+X_{ii}d(Y_{ij}),P_{j}],\cdots,P_{j}]\\
&=d(X_{ii})Y_{ij}+X_{ii}d(Y_{ij}),
\end{aligned}\eqno(7)
$$
where $P_{j}=P$ if $j=1$, else $P_{j}=Q$. Similarly we have
$$d(X_{ij}Y_{jj})=d(X_{ij})Y_{jj}+X_{ij}d(Y_{jj}).\eqno(8)$$

From Eq.(7), it is clearly that
$$d(X_{ii}Y_{ii}M_{ij})=d(X_{ii}Y_{ii})M_{ij}+X_{ii}Y_{ii}d(M_{ij}).$$
On the other hand,
$$\begin{aligned}
d(X_{ii}Y_{ii}M_{ij})=&d(X_{ii})Y_{ii}M_{ij}+X_{ii}d(Y_{ii}M_{ij})\\
=&d(X_{ii})Y_{ii}M_{ij}+X_{ii}d(Y_{ii})M_{ij}+X_{ii}Y_{ii}d(M_{ij}).
\end{aligned}$$
Comparing with the two expressions, we obtain
$$(d(X_{ii}Y_{ii})-d(X_{ii})Y_{ii}-X_{ii}d(Y_{ii}))M_{ij}=0.$$
By Lemma \ref {xxsec2.1} (1), we have
$$D(X_{ii}Y_{ii})=D(X_{ii})Y_{ii}+X_{ii}D(Y_{ii}).  \eqno(9)$$

Noting that $d$ is additive, there exists a central element
$Z\in\mathcal{Z}_{\mathcal{A}}$ such that
$$\begin{aligned}
&d(X_{12}Y_{21})-d(Y_{21}X_{12})=d(X_{12}Y_{21}-Y_{21}X_{12})=d([X_{12},Y_{21}])\\
=&\varphi([X_{12},Y_{21}])+Z=\varphi([[\cdots[X_{12},Q],\cdots,Q],Y_{21}])+Z\\
=&[[\cdots[\varphi(X_{12}),Q],\cdots,Q],Y_{21}]+[[\cdots[X_{12},Q],\cdots,Q],\varphi(Y_{21})]+Z\\
=&[d(X_{12}),Y_{21}]+[X_{12},d(Y_{21})]+Z\\
=&d(X_{12})Y_{21}+X_{12}d(Y_{21})-Y_{21}d(X_{12})-d(Y_{21})X_{12}+Z.
\end{aligned}$$
Hence $d(X_{12}Y_{21})-d(X_{12})Y_{21}-X_{12}d(Y_{21})\in
P\mathcal{Z}_{\mathcal{A}}$ and
$d(Y_{21}X_{12})-Y_{21}d(X_{12})-d(Y_{21})X_{12}\in
Q\mathcal{Z}_{\mathcal{A}}$.

Let $X=\sum_{1\leq i,j\leq 2}X_{ij}$, $Y=\sum_{1\leq
i,j\leq2}Y_{ij}$, where $X_{ij},Y_{ij}\in\mathcal{A}_{ij}$. A direct
computation shows that
$$\begin{aligned}
d(XY)&=d(X_{11}Y_{11})+d(X_{11}Y_{12})+d(X_{12}Y_{21})+d(X_{12}Y_{22}) \\
&\quad +d(X_{21}Y_{11})+d(X_{21}Y_{12})+d(X_{22}Y_{21})+d(X_{22}Y_{22}) \\
&=d(X)Y+Xd(Y)+PZ_{1}+QZ_{2}
\end{aligned}
$$
where $Z_{1},Z_{2}\in \mathcal{Z}_{\mathcal{A}}$.
Define a map $\theta: \mathcal{A}\times\mathcal{A}\longrightarrow
P\mathcal{Z}_{\mathcal{A}}\oplus Q\mathcal{Z}_{\mathcal{A}}$ by
$$\theta(X,Y)=d(XY)-d(X)Y-Xd(Y).$$
It is clear that $\theta$ is a bi-additive mapping. So that for any
$X\in \mathcal{A}_{11}\oplus\mathcal{A}_{22}$ and $Y\in\mathcal{A}$,
there is $\theta(X,Y)=\theta(Y,X)=0$ by Eq.(7)-Eq.(9).

We only need to show that $\theta(X,Y)\equiv 0$. In fact, for any
$X,Y,Z\in\mathcal{A}$,
$$\begin{aligned}
d(XYZ)&=d((XY)Z)=d(XY)Z+XYd(Z)+\theta(XY,Z)\\
&=d(X)YZ+Xd(Y)Z+XYd(Z)+\theta(X,Y)Z+\theta(XY,Z).
\end{aligned}$$
On the other hand,
$$\begin{aligned}
d(XYZ)&=d(X(YZ))=d(X)YZ+Xd(YZ)+\theta(X,YZ)\\
&=d(X)YZ+Xd(Y)Z+XYd(Z)+X\theta(Y,Z)+\theta(X,YZ).
\end{aligned}$$
Hence $$\theta(X,Y)Z+\theta(XY,Z)=X\theta(Y,Z)+\theta(X,YZ).$$
Therefore $\theta$ is a Hochschild 2-cocycle. Taking $Z=Z_{ii}\in\mathcal{A}_{ii}$, we have
$\theta(X,Y)Z_{ii}=\theta(X,YZ_{ii})$. Assume
$\theta(X,Y)=PZ_{1}+QZ_{2}$ for some
$Z_{1},Z_{2}\in\mathcal{Z}_{\mathcal{A}}$. Then
$\theta(X,Y)Z_{11}=PZ_{1}Z_{11}=\theta(X,YZ_{11})\in
P\mathcal{Z}_{\mathcal{A}}\oplus Q\mathcal{Z}_{\mathcal{A}}$. Thus
$PZ_{1}Z_{11}\in
P\mathcal{Z}_{\mathcal{A}}=\mathcal{Z}_{\mathcal{A}_{11}}$. Since
$\mathcal{A}$ has no central abelian summands, we know
$\mathcal{A}_{11}$ has no central abelian summands too. Then Lemma
\ref{xxsec2.1} (4) shows that $PZ_{1}=0$. Similarly, we can prove
that $QZ_{2}=0$. Hence $\theta(X,Y)=0$.

\textbf{Step 3.} We show that
$f([\cdots[[X_{1},X_{2}],X_{3}],\cdots,X_{n}])=0$ for all
$X_{i}\in\mathcal{A}$.

In fact,
$$\begin{aligned}
&f([\cdots[[X_{1},X_{2}],X_{3}],\cdots,X_{n}])\\
=&\varphi([\cdots[[X_{1},X_{2}],X_{3}],\cdots,X_{n}])-d([\cdots[[X_{1},X_{2}],X_{3}],\cdots,X_{n}])\\
=&[\cdots[[\varphi(X_{1}),X_{2}],X_{3}],\cdots,X_{n}]+[\cdots[[X_{1},\varphi(X_{2})],X_{3}],\cdots,X_{n}]+\cdots\\
&+[\cdots[[X_{1},X_{2}],X_{3}],\cdots,\varphi(X_{n})]-d([\cdots[[X_{1},X_{2}],X_{3}],\cdots,X_{n}])\\
=&[\cdots[[d(X_{1}),X_{2}],X_{3}],\cdots,X_{n}]+[\cdots[[X_{1},d(X_{2})],X_{3}],\cdots,X_{n}]+\cdots\\
&+[\cdots[[X_{1},X_{2}],X_{3}],\cdots,d(X_{n})]-d([\cdots[[X_{1},X_{2}],X_{3}],\cdots,X_{n}])=0.
\end{aligned}$$
\qed

\vspace{6pt}
\noindent{\bf Acknowledgements.}
The the first author of this work is supported
by a research foundation of Huaqiao University (Grant No. 10BS323).
The second author is partially supported by the Tianyuan 
Mathematics Foundation of China (Grant No. 11126331).
The work of the third author is partially supported by the National
Nature Science Foundation of China (Grant No. 10871023).

\end{document}